\documentclass[a4paper,twoside,fleqn]{article}
\usepackage{amsfonts}
\usepackage[dvips]{graphics}
\usepackage{amssymb}
\usepackage[reqno]{amsmath}
\usepackage{amsthm}
\usepackage{makeidx}
\usepackage{color}
\usepackage{graphicx}
\usepackage{subfigure}
\usepackage{multicol}
\usepackage{fancyhdr}
\usepackage[colorlinks,linkcolor=blue,anchorcolor=blue,citecolor=blue]{hyperref}

 \newtheorem{theorem}{\sc\bf Theorem}[section]

 \newtheorem{lemma}[theorem]{\sc\bf Lemma}
 
 \newtheorem{definition}[theorem]{\sc\bf Definition}

 \newtheorem{example}[theorem]{\sc\bf Example}
  
  \numberwithin{equation}{section}

\makeatletter
\def\@cite#1#2{#1\if@tempswa , #2\fi}
\makeatother

\title{{\bf A note on property $(gb)$ and perturbations}
\thanks{This work has been supported by National Natural Science Foundation of China (11171066), Specialized Research Fund for the Doctoral
Program of Higher Education (2010350311001, 20113503120003), Natural
Science Foundation of Fujian Province (2009J01005, 2011J05002) and
Foundation of the Education Department of Fujian Province (JB10042,
JA08036).} }

\author{Qingping \textsc{Zeng}\thanks{Email address: zqpping2003@163.com.}
 \quad Huaijie \textsc{Zhong}
\\ \small (School of Mathematics and Computer Science, Fujian Normal
University, Fuzhou 350007, P.R. China) }

\begin{document}
\date{}
\maketitle

\large

\begin{quote}
 {\bf Abstract:} ~An operator $T \in \mathcal{B}(X)$ defined on a
 Banach space $X$ satisfies property $(gb)$ if the complement in the
 approximate point spectrum $\sigma_{a}(T)$ of the upper semi-B-Weyl
 spectrum $\sigma_{SBF_{+}^{-}}(T)$ coincides with the set $\Pi(T)$ of all poles of the resolvent of
 $T$. In this note we continue to study property
 $(gb)$ and the stability of it, for a bounded linear operator $T$ acting on a Banach space,
 under perturbations by nilpotent operators, by finite rank operators, by
 quasi-nilpotent operators commuting with $T$. Two counterexamples
 show
 that property $(gb)$ in general is not preserved under commuting quasi-nilpotent
 perturbations or commuting finite rank perturbations. 
\end{quote}

\section{Introduction}
 \quad\,~Throughout this note, let $\mathcal{B}(X)$ denote the Banach
algebra of all bounded linear operators acting on an infinite
dimensional
 complex Banach space $X$, and $\mathcal{F}(X)$ denote its ideal of finite rank operators on $X$.
For an operator $T \in \mathcal{B}(X)$, let $T^{*}$ denote its dual,
$\mathcal {N}(T)$ its kernel, $\alpha(T)$ its nullity, $\mathcal
{R}(T)$ its range, $\beta(T)$ its defect, $\sigma(T)$ its spectrum
and $\sigma_{a}(T)$ its approximate point spectrum. If the range
$\mathcal {R}(T)$ is closed and $\alpha(T) < \infty$ (resp.
$\beta(T) < \infty$), then $T$ is said to be $upper$
$semi$-$Fredholm$ (resp. $lower$ $semi$-$Fredholm$).  If $T \in
\mathcal{B}(X)$ is both upper and lower semi-Fredholm, then $T$ is
said to be $Fredholm$. If $T \in \mathcal{B}(X)$ is either upper or
lower semi-Fredholm, then $T$ is said to be $semi$-$Fredholm$, and
its index is defined by
\begin{upshape}ind\end{upshape}$(T)$ = $\alpha(T)-\beta(T)$.
The $upper$ $semi$-$Weyl$ $operators$ are defined as the class of
upper semi-Fredholm operators with index less than or equal to zero,
while $Weyl$ $operators$ are defined as the class of Fredholm
operators of index zero. These classes of operators generate the
following spectra: the $Weyl$ $spectrum$ defined by
$$\sigma_{W}(T):= \{ \lambda \in \mathbb{C}: T - \lambda I \makebox{ is not a Weyl operator} \},$$
the $upper$ $semi$-$Weyl$ $spectrum$ (in literature called also
$Weyl$ $essential$ $approximate$ $point$ $spectrum$) defined by
$$\sigma_{SF_{+}^{-}}(T):= \{ \lambda \in \mathbb{C}: T - \lambda I \makebox{ is not a upper semi-Weyl operator}\}.$$

 Recall that the
$descent$ and the $ascent$ of $T \in \mathcal{B}(X)$ are
 $dsc(T)= \inf \{n \in \mathbf{\mathbb{N}}:\mathcal {R}(T^{n})= \mathcal {R}(T^{n+1})\}$
 and $asc(T)=\inf \{n \in \mathbf{\mathbb{N}}:\mathcal {N}(T^{n})= \mathcal {N}(T^{n+1})\}$,
 respectively (the infimum of an empty set is defined to be $\infty$).
 If $asc(T) < \infty$ and $\mathcal {R}(T^{asc(T)+1})$ is closed,
then $T$ is said to be $left$ $Drazin$ $invertible$. If $dsc(T) <
\infty $ and $\mathcal {R}(T^{dsc(T)})$ is closed, then $T$ is said
to be $right$ $Drazin$ $invertible$. If $asc(T) = dsc(T) < \infty $,
then $T$ is said to be $Drazin$ $invertible$. Clearly, $T \in
\mathcal{B}(X)$ is both left and right Drazin invertible if and only
if $T$ is Drazin invertible. An operator $T \in \mathcal{B}(X)$ is
called $upper$ $semi$-$Browder$ if it is a upper semi-Fredholm
operator with finite ascent, while $T$ is called
 $Browder$ if it is a Fredholm operator of finite ascent and descent. The $Browder$
 $spectrum$ of $T \in \mathcal{B}(X)$ is defined by
$$\sigma_{B}(T):= \{ \lambda \in \mathbb{C}: T - \lambda I \makebox{ is not a Browder operator} \},$$
the $upper$ $semi$-$Browder$ $spectrum$ (in literature called also
$Browder$ $essential$ $approximate$ $point$ $spectrum$) is defined
by
$$\sigma_{UB}(T):= \{ \lambda \in \mathbb{C}: T - \lambda I \makebox{ is not a upper semi-Browder operator}\}.$$
An operator $T \in \mathcal{B}(X)$ is called $Riesz$ if its
essential spectrum $\sigma_e(T):= \{ \lambda \in \mathbb{C}: T -
\lambda I \makebox{ is not Fredholm} \}=\{0\}$.

Suppose that $T \in \mathcal{B}(X)$ and that $R \in \mathcal{B}(X)$
is a Riesz operator commuting with $T$. Then it follows from
[\cite{Tylli}, Proposition 5] and [\cite{Rakocevic}, Theorem 1] that
 \begin{equation} {\label{eq 1.1}}
 \qquad\qquad\qquad\qquad\qquad\quad \ \     \sigma_{SF_{+}^{-}}(T+R) = \sigma_{SF_{+}^{-}}(T);
\end{equation} \vspace{-4mm}
\begin{equation} {\label{eq 1.2}}
 \qquad\qquad\qquad\qquad\qquad\qquad \ \     \sigma_{W}(T+R) =
 \sigma_{W}(T);
\end{equation} \vspace{-4mm}
\begin{equation} {\label{eq 1.3}}
 \qquad\qquad\qquad\qquad\qquad\quad \ \     \sigma_{UB}(T+R) = \sigma_{UB}(T);
\end{equation} \vspace{-4mm}
\begin{equation} {\label{eq 1.4}}
 \qquad\qquad\qquad\qquad\qquad\qquad \ \     \sigma_{B}(T+R) =
 \sigma_{B}(T).
\end{equation}

For each integer $n$, define $T_{n}$ to be the restriction of $T$ to
$\mathcal{R}(T^{n})$ viewed as the map from $\mathcal{R}(T^{n})$
into $\mathcal{R}(T^{n})$ (in particular $T_{0} = T $). If there
exists $n \in \mathbb{N}$ such that $\mathcal {R}(T^{n})$ is closed
and $T_{n}$ is upper semi-Fredholm, then $T$ is called $upper$
$semi$-$B$-$Fredholm$. It follows from [\cite{Berkani
semi-B-fredholm}, Proposition 2.1] that if there exists $n \in
\mathbb{N}$ such that $\mathcal {R}(T^{n})$ is closed and $T_{n}$ is
upper semi-Fredholm, then $\mathcal {R}(T^{m})$ is closed, $T_{m}$
is upper semi-Fredholm and
\begin{upshape}ind\end{upshape}$(T_{m})$ = \begin{upshape}ind\end{upshape}$(T_{n})$ for all $m \geq
n$. This enables us to define the index of a upper semi-B-Fredholm
operator $T$ as the index of the upper semi-Fredholm operator
$T_{n}$, where $n$ is an integer satisfying $\mathcal {R}(T^{n})$ is
closed and $T_{n}$ is upper semi-Fredholm. An operator $T \in
\mathcal {B}(X)$ is called $upper$ $semi$-$B$-$Weyl$ if $T$ is upper
semi-B-Fredholm and
\begin{upshape}ind\end{upshape}$(T) \leq 0$.

For $T \in \mathcal{B}(X)$, let us define the $left$ $Drazin$
$spectrum$, the $Drazin$ $spectrum$ and the $upper$
$semi$-$B$-$Weyl$ $spectrum$ of $T$  as follows respectively:
$$ \sigma_{LD}(T) := \{ \lambda \in \mathbb{C}: T - \lambda I \makebox{ is not a left Drazin invertible operator} \};$$
$$ \sigma_{D}(T):= \{ \lambda \in \mathbb{C}: T - \lambda I \makebox{ is not a Drazin invertible operator} \};$$
$$ \sigma_{SBF_{+}^{-}}(T) := \{ \lambda \in \mathbb{C}: T - \lambda I \makebox{ is not a upper semi-B-Weyl
operator} \}.$$

Let $\Pi(T)$ denote the set of all poles of $T$. We say that
$\lambda \in \sigma_{a}(T)$ is a left pole of $T$ if $T-\lambda I$
is left Drazin invertible. Let $\Pi_{a}(T)$ denote the set of all
left poles of $T$. It is well know that $\Pi(T)=\sigma(T) \backslash
\sigma_{D}(T) = \makebox{iso}\sigma(T) \backslash \sigma_{D}(T)$ and
$\Pi_{a}(T)=\sigma_{a}(T) \backslash
\sigma_{LD}(T)=\makebox{iso}\sigma_{a}(T) \backslash \sigma_{LD}(T)
$. Here and henceforth, for $A \subseteq \mathbb{C}$,
$\makebox{iso}A$ is the set of isolated points of $A$. An operator
$T \in \mathcal{B}(X)$ is called $a$-$polaroid$ if
$\makebox{iso}\sigma_{a}(T)= \varnothing $ or every isolated point
of $\sigma_{a}(T)$ is a left pole of $T$.

Following Harte and Lee [\cite{Harte-Lee}] we say that $T \in
\mathcal{B}(X)$ satisfies Browder's theorem if
$\sigma_{W}(T)=\sigma_{B}(T)$. While, according to Djordjevi\'c and
Han [\cite{Djordjevic-Han}], we say that $T$ satisfies a-Browder's
theorem if $\sigma_{SF_{+}^{-}}(T)=\sigma_{UB}(T)$.

The following two variants of Browder's theorem have been introduced
by Berkani and Zariouh [\cite{Berkani-Zariouh Extended}] and Berkani
and Koliha [\cite{Berkani-Koliha}], respectively.

\begin {definition}${\label{1.1}}$ \begin{upshape}
An operator $T \in \mathcal{B}(X)$ is said to possess property
$(gb)$ if $$\sigma_{a}(T) \backslash \sigma_{SBF_{+}^{-}}(T) =
\Pi(T).$$ While $T \in \mathcal{B}(X)$ is said to satisfy
generalized a-Browder's theorem if $$\sigma_{a}(T) \backslash
\sigma_{SBF_{+}^{-}}(T) = \Pi_{a}(T).$$
\end{upshape}
\end{definition}

From formulas (\ref{eq 1.1})--(\ref{eq 1.4}), it follows immediately
that Browder's theorem and a-Browder's theorem are preserved under
commuting Riesz perturbations.  It is proved in [\cite{AmouchM
ZguittiH equivalence}, Theorem 2.2] that generalized a-Browder's
theorem is equivalent to a-Browder's theorem. Hence, generalized
a-Browder's theorem is stable under commuting Riesz perturbations.
That is, if $T \in \mathcal{B}(X)$ satisfies generalized a-Browder's
theorem and $R$ is a Riesz operator commuting with $T$, then $T+R$
satisfies generalized a-Browder's theorem.

The single-valued extension property was introduced by Dunford in
[\cite{Dunford 1},\cite{Dunford 2}] and has an important role in
local spectral theory and Fredholm theory, see the recent monographs
[\cite{Aiena}] by Aiena and [\cite{Laursen-Neumann}] by Laursen and
Neumann.

\begin {definition}${\label{1.1}}$ \begin{upshape}
An operator $T \in \mathcal{B}(X)$ is said to have the single-valued
extension property at $\lambda_{0} \in \mathbb{C}$ (SVEP at
$\lambda_{0}$ for brevity), if for every open neighborhood $U$ of
$\lambda_{0}$ the only analytic function $f:U \rightarrow X$ which
satisfies the equation $(\lambda I-T)f(\lambda) = 0$ for all $
\lambda \in U$ is the function $f (\lambda ) \equiv 0 $.

Let $S(T):= \{\lambda \in \mathbb{C}: T \makebox{ does not have the
SVEP at } \lambda \}$. An operator $T \in \mathcal{B}(X)$ is said to
have SVEP if $S(T)=\varnothing$.
\end{upshape}
\end{definition}

In this note we continue the study of property $(gb)$ which is
studied in some recent papers [\cite{Berkani-Zariouh
Extended},\cite{Berkani-Zariouh New Extended},\cite{Rashid}]. We
show that property $(gb)$ is satisfied by an operator $T$ satisfying
$S(T^{*}) \subseteq \sigma_{SBF_{+}^{-}}(T)$. We give a revised
proof of [\cite{Rashid}, Theorem 3.10] to prove that property $(gb)$
is preserved under commuting nilpotent perturbations. We show also
that if $T \in \mathcal{B}(X)$ satisfies $S(T^*) \subseteq
\sigma_{SBF_{+}^{-}}(T)$ and $F$ is a finite rank operator commuting
with $T$, then $T+F$ satisfies property $(gb)$. We show that if $T
\in \mathcal{B}(X)$ is a a-polaroid operator satisfying property
$(gb)$ and $Q$ is a quasi-nilpotent operator commuting with $T$,
then $T+Q$ satisfies property $(gb)$. Two counterexamples are also
given to show
 that property $(gb)$ in general is not preserved under commuting quasi-nilpotent
 perturbations or commuting finite rank perturbations. These results
 improve and revise some recent results of Rashid in
 [\cite{Rashid}].

\section{Main results}

\quad\,~We begin with the following lemmas.

\begin {lemma}${\label{2.1}}$ \begin{upshape} ([\cite{Berkani-Zariouh Extended}, Corollary 2.9])   \end{upshape}
An operator $T \in \mathcal{B}(X)$ possesses property $(gb)$ if and
only if $T$ satisfies generalized a-Browder's theorem and
$\Pi(T)=\Pi_{a}(T)$.
\end{lemma}

\begin {lemma}${\label{2.2}}$ If the equality
$\sigma_{SBF_{+}^{-}}(T)=\sigma_{D}(T)$ holds for $T \in
\mathcal{B}(X)$, then $T$ possesses property $(gb)$.
\end{lemma}

\begin{proof} Suppose that $\sigma_{SBF_{+}^{-}}(T)=\sigma_{D}(T)$.
If $\lambda \in  \sigma_{a}(T) \backslash \sigma_{SBF_{+}^{-}}(T)$,
then $\lambda \in \sigma_{a}(T) \backslash \sigma_{D}(T)$ $
\subseteq \Pi(T)$. This implies that $\sigma_{a}(T) \backslash
\sigma_{SBF_{+}^{-}}(T) = \Pi(T)$. Since $\Pi(T) \subseteq
\sigma_{a}(T) \backslash \sigma_{SBF_{+}^{-}}(T)$ is always true,
$\sigma_{a}(T) \backslash \sigma_{SBF_{+}^{-}}(T) = \Pi(T)$, i.e.
$T$ possesses property $(gb)$.
\end{proof}

\begin {lemma}${\label{2.3}}$ If $T \in
\mathcal{B}(X)$, then $\sigma_{SBF_{+}^{-}}(T) \cup S(T^*) =
\sigma_{D}(T)$.
\end{lemma}

\begin{proof} Let $\lambda \notin \sigma_{SBF_{+}^{-}}(T) \cup
S(T^*)$. Then $T - \lambda$ is a upper semi-Weyl operator and $T^*$
has SVEP at $\lambda$. Thus $T - \lambda$ is a upper semi-B-Fredholm
operator and $\makebox{ind}(T - \lambda) \leq 0$. Hence there exists
$n \in \mathbb{N}$ such that $\mathcal {R}((T-\lambda)^{n})$ is
closed, $(T - \lambda)_{n}$ is a upper semi-Fredholm operator and
$\makebox{ind}(T - \lambda)_{n} \leq 0$. By
[\cite{Aiena2007quasifredholm}, Theorem 2.11], $dsc(T - \lambda) <
\infty$. Thus $dsc(T - \lambda)_{n} < \infty$, by [\cite{Aiena},
Theorem 3.4(ii)], $\makebox{ind}(T - \lambda)_{n} \geq 0$. By
[\cite{Aiena}, Theorem 3.4(iv)], $asc(T - \lambda)_{n} = dsc(T -
\lambda)_{n} < \infty$. Consequently, $(T - \lambda)_{n}$ is a
Browder operator. Thus by [\cite{Aiena-Biondi-Carpintero}, Theorem
2.9] we then conclude that $T - \lambda$ is Drazin invertible, i.e.
$\lambda \notin \sigma_{D}(T)$. Hence $\sigma_{D}(T) \subseteq
\sigma_{SBF_{+}^{-}}(T) \cup S(T^*)$. Since the reverse inclusion
obviously holds, we get $\sigma_{SBF_{+}^{-}}(T) \cup S(T^*) =
\sigma_{D}(T)$.
\end{proof}

\begin {theorem}${\label{2.4}}$ If $T \in
\mathcal{B}(X)$ satisfies $S(T^*) \subseteq
\sigma_{SBF_{+}^{-}}(T)$, then $T$ possesses property $(gb)$. In
particular, if $T^*$ has SVEP, then $T$ possesses property $(gb)$.
\end{theorem}

\begin{proof} Suppose that $S(T^*) \subseteq
\sigma_{SBF_{+}^{-}}(T)$. Then by Lemma \ref{2.3}, we get
$\sigma_{SBF_{+}^{-}}(T) = \sigma_{D}(T)$. Consequently, by Lemma
\ref{2.2}, $T$ possesses property $(gb)$. If $T^*$ has SVEP, then
$S(T^*)=\varnothing$, the conclusion follows immediately.
\end{proof}

The following example shows that the converse of Theorem \ref{2.4}
is not true.

\begin {example}${\label{2.5}}$ \begin{upshape}
Let $X$ be the Hilbert space $l_{2}(\mathbb{N})$ and let $T:
l_{2}(\mathbb{N}) \longrightarrow l_{2}(\mathbb{N})$ be the
unilateral right shift operator defined by
$$T(x_{1},x_{2},\cdots )=(0,x_{1},x_{2}, \cdots ) \makebox{\ \ \  for all \ } (x_{n}) \in
l_{2}(\mathbb{N}).$$ Then,
$$\sigma_{a}(T) =  \{\lambda \in
\mathbb{C}: |\lambda| = 1 \},$$
$$\sigma_{SBF_{+}^{-}}(T) =  \{\lambda \in
\mathbb{C}: |\lambda| = 1 \}$$ and $$\Pi(T) = \varnothing.$$ Hence
$\sigma_{a}(T) \backslash \sigma_{SBF_{+}^{-}}(T) = \Pi(T)$, i.e.
$T$ possesses property $(gb)$.

But $S(T^*)= \{\lambda \in \mathbb{C}: 0 \leq |\lambda| < 1 \}
\nsubseteq \{\lambda \in \mathbb{C}: |\lambda| = 1 \} =
\sigma_{SBF_{+}^{-}}(T).$

  \end{upshape}
\end{example}

The next theorem had been established in [\cite{Rashid}, Theorem
3.10], but its proof was not so clear. Hence we give a revised proof
of it.

\begin {theorem}${\label{2.6}}$ If $T \in \mathcal{B}(X)$ satisfies property
$(gb)$ and $N$ is a nilpotent operator that commutes with $T$, then
$T+N$ satisfies property $(gb)$.
\end{theorem}

\begin{proof} Suppose that $T \in \mathcal{B}(X)$ satisfies property
$(gb)$ and $N$ is a nilpotent operator that commutes with $T$. By
Lemma \ref{2.1}, $T$ satisfies generalized a-Browder's theorem and
$\Pi(T)=\Pi_{a}(T)$. Hence $T+N$ satisfies generalized a-Browder's
theorem. By [\cite{Mbekhta-Muller 9}], $\sigma(T+N) = \sigma(T)$ and
$\sigma_{a}(T+N) = \sigma_{a}(T)$. Hence, by [\cite{Kaashoek-Lay},
Theorem 2.2] and [\cite{Bel-Burgos-Oudghiri 3}, Theorem 3.2], we
have that $\Pi(T+N)=\sigma(T+N) \backslash \sigma_{D}(T+N) =
\sigma(T) \backslash \sigma_{D}(T)=\Pi(T)=\Pi_{a}(T) = \sigma_{a}(T)
\backslash \sigma_{LD}(T)=\sigma_{a}(T+N) \backslash
\sigma_{LD}(T+N) = \Pi_{a}(T+N)$. By Lemma \ref{2.1} again, $T+N$
satisfies property $(gb)$.
\end{proof}

The following example, which is a revised version of [\cite{Rashid},
Example 3.11], shows that the hypothesis of commutativity in Theorem
\ref{2.6} is crucial.

\begin {example}${\label{2.7}}$ \begin{upshape}
Let $T: l_{2}(\mathbb{N}) \longrightarrow l_{2}(\mathbb{N})$ be the
unilateral right shift operator defined by
$$T(x_{1},x_{2},\cdots )=(0,x_{1},x_{2}, \cdots ) \makebox{\ \ \  for all \ } (x_{n}) \in
l_{2}(\mathbb{N}).$$ Let $N: l_{2}(\mathbb{N}) \longrightarrow
l_{2}(\mathbb{N})$ be a nilpotent operator with rank one defined by
$$N(x_{1},x_{2},\cdots )=(0,-x_{1},0, \cdots ) \makebox{\ \ \  for all \ } (x_{n}) \in
l_{2}(\mathbb{N}).$$ Then $TN \neq NT$. Moreover,
$$\sigma(T) = \{\lambda \in
\mathbb{C}: 0 \leq |\lambda| \leq 1 \},$$
$$\sigma_{a}(T) = \{\lambda \in
\mathbb{C}: |\lambda| = 1 \} ,$$
$$\sigma(T+N) =  \{\lambda \in \mathbb{C}: 0 \leq
|\lambda| \leq 1 \}$$ and
$$\sigma_{a}(T+N) = \{\lambda \in
\mathbb{C}: |\lambda| = 1 \} \cup \{ 0 \}.$$ It follows that
$\Pi_{a}(T)=\Pi(T)= \varnothing$ and $\{0 \}= \Pi_{a}(T+N) \neq
\Pi(T+N)= \varnothing.$ Hence by Lemma \ref{2.1}, $T+N$ does not
satisfy property $(gb)$. But since $T$ has SVEP, $T$ satisfies
a-Browder's theorem or equivalently, by [\cite{AmouchM ZguittiH
equivalence}, Theorem 2.2], $T$ satisfies generalized a-Browder's
theorem. Therefore by Lemma \ref{2.1} again, $T$ satisfies property
$(gb)$.

  \end{upshape}
\end{example}

To continue the discussion of this note, we recall some classical
definitions. Using the isomorphism $X/\mathcal {N}(T^{d}) \approx
\mathcal {R}(T^{d})$ and
    following [\cite{Grabiner 9}], a topology on $\mathcal {R}(T^{d})$ is
    defined as follows.

\begin{definition} ${\label{2.8}}$ \begin{upshape} Let $T \in \mathcal{B}(X)$. For every $d \in \mathbf{\mathbb{N}},$
the operator range topological on $\mathcal {R}(T^{d})$ is defined
by the norm $||\small{\cdot}||_{\mathcal {R}(T^{d})}$ such that for
all $y \in \mathcal {R}(T^{d})$,
$$||y||_{\mathcal {R}(T^{d})} := \inf\{||x||: x \in X,y=T^{d}x\}.$$
  \end{upshape}
\end{definition}
 For a detailed discussion of operator ranges and their topologies,
 we refer the reader to [\cite{Fillmore-Williams 7}] and [\cite{Grabiner 8}].

\begin {definition}${\label{2.9}}$ \begin{upshape} Let $T \in
    \mathcal{B}(X)$ and let $d \in \mathbf{\mathbb{N}}$. Then $T$ has
    $uniform$ $descent$ for $n \geq d$ if $k_{n}(T)=0$ for all $n \geq d$. If in addition $\mathcal {R}(T^{n})$
    is closed in the operator range topology of $\mathcal {R}(T^{d})$ for all $n \geq d$$,$ then we say that $T$
    has $eventual$ $topological$ $uniform$
    $descent$$,$ and$,$ more precisely$,$ that $T$ has $topological$ $uniform$
    $descent$ $for$ $n \geq d$.
    \end{upshape}

\end{definition}

Operators with eventual topological uniform descent are introduced
by Grabiner in [\cite{Grabiner 9}]. It includes many classes of
operators introduced in the Introduction of this note, such as upper
semi-B-Fredholm operators, left Drazin invertible operators, Drazin
invertible operators, and so on. It also includes many other classes
of operators such as operators of Kato type, quasi-Fredholm
operators, operators with finite descent and operators with finite
essential descent, and so on. A very detailed and far-reaching
account of these notations can be seen in
[\cite{Aiena},\cite{Berkani},\cite{Mbekhta-Muller 9}]. Especially,
operators which have topological uniform descent for $n \geq 0$ are
precisely the $semi$-$regular$ operators studied by Mbekhta in
[\cite{Mbekhta}]. Discussions of operators with eventual topological
uniform descent may be found in
[\cite{Berkani-Castro-Djordjevic},\cite{Cao},\cite{Grabiner
9},\cite{Jiang-Zhong-Zeng},\cite{Zeng-Zhong-Wu}].

\begin {lemma}${\label{2.10}}$ If $T \in \mathcal{B}(X)$ and $F$ is a finite rank operator commuting with
$T$, then

 $(1)$ $\sigma_{SBF_{+}^{-}}(T+F) =\sigma_{SBF_{+}^{-}}(T)$;

 $(2)$ $\sigma_{D}(T+F) =\sigma_{D}(T)$.

\end{lemma}

\begin{proof}
 $(1)$ Without loss of generality, we need only to show that
 $0 \notin \sigma_{SBF_{+}^{-}}(T+F)$ if and only if  $0 \notin
 \sigma_{SBF_{+}^{-}}(T)$. By symmetry,
 it suffices to prove that $0 \notin \sigma_{SBF_{+}^{-}}(T+F)$
 if $0 \notin \sigma_{SBF_{+}^{-}}(T)$.

 Suppose that $0 \notin \sigma_{SBF_{+}^{-}}(T)$. Then $T$ is a upper
 semi-B-Fredholm operator and $\makebox{ind}(T) \leq 0$. Hence it
 follows from [\cite{Berkani}, Theorem 3.6] and [\cite{Bel-Burgos-Oudghiri 3}, Theorem
 3.2] that $T+F$ is also a upper semi-B-Fredholm operator. Thus by [\cite{Grabiner 9}, Theorem
 5.8], $\makebox{ind}(T+F)=\makebox{ind}(T) \leq 0$. Consequently, $T+F$ is a semi-B-Weyl operator, i.e.
 $0 \notin \sigma_{SBF_{+}^{-}}(T)$, and this completes the proof of $(1)$.

 $(2)$ Noting that an operator is Drazin invertible if and only if it is of finite ascent and finite descent,
  the conclusion follows from [\cite{Kaashoek-Lay}, Theorem 2.2].
\end{proof}

\begin {theorem}${\label{2.11}}$ If $T \in \mathcal{B}(X)$ satisfies $S(T^*) \subseteq
\sigma_{SBF_{+}^{-}}(T)$ and $F$ is a finite rank operator commuting
with $T$, then $T+F$ satisfies property $(gb)$.

\end{theorem}

\begin{proof} Since $F$ is a finite rank operator commuting with
$T$, by Lemma \ref{2.10}, $\sigma_{SBF_{+}^{-}}(T+F)
=\sigma_{SBF_{+}^{-}}(T)$ and $\sigma_{D}(T+F) =\sigma_{D}(T)$.
Since $S(T^*) \subseteq \sigma_{SBF_{+}^{-}}(T)$, by Lemma
\ref{2.3}, $\sigma_{SBF_{+}^{-}}(T)=\sigma_{D}(T).$ Thus,
$\sigma_{SBF_{+}^{-}}(T+F)=\sigma_{D}(T+F)$. By Lemma \ref{2.2},
$T+F$ satisfies property $(gb)$.
\end{proof}

The following example illustrates that property $(gb)$ in general is
not preserved under commuting finite rank perturbations.

\begin {example}${\label{2.12}}$ \begin{upshape}
Let $U: l_{2}(\mathbb{N}) \longrightarrow l_{2}(\mathbb{N})$ be the
unilateral right shift operator defined by
$$U(x_{1},x_{2},\cdots )=(0,x_{1},x_{2}, \cdots ) \makebox{\ \ \  for all \ } (x_{n}) \in
l_{2}(\mathbb{N}).$$ For fixed $0 < \varepsilon < 1$, let
$F_{\varepsilon}: l_{2}(\mathbb{N}) \longrightarrow
l_{2}(\mathbb{N})$ be a finite rank operator defined by
$$F_{\varepsilon}(x_{1},x_{2},\cdots )=(-\varepsilon x_{1},0,0, \cdots ) \makebox{\ \ \  for all \ } (x_{n}) \in
l_{2}(\mathbb{N}).$$ We consider the operators $T$ and $F$ defined
by $T=U \oplus I$ and $F=0 \oplus F_{\varepsilon}$, respectively.
Then $F$ is a finite rank operator and $TF=FT$. Moreover,
$$\sigma(T) = \sigma(U) \cup \sigma(I) = \{\lambda \in
\mathbb{C}: 0 \leq |\lambda| \leq 1 \},$$
$$\sigma_{a}(T) = \sigma_{a}(U) \cup \sigma_{a}(I) = \{\lambda \in
\mathbb{C}: |\lambda| = 1 \} ,$$
$$\sigma(T+F) = \sigma(U) \cup
\sigma(I+F_{\varepsilon}) = \{\lambda \in \mathbb{C}: 0 \leq
|\lambda| \leq 1 \}$$ and
$$\sigma_{a}(T+F) = \sigma_{a}(U) \cup \sigma_{a}(I+F_{\varepsilon}) = \{\lambda \in
\mathbb{C}: |\lambda| = 1 \} \cup \{ 1 - \varepsilon \}.$$ It
follows that $\Pi_{a}(T)=\Pi(T)= \varnothing$ and $\{1 - \varepsilon
\}= \Pi_{a}(T+F) \neq \Pi(T+F)= \varnothing.$ Hence by Lemma
\ref{2.1}, $T+F$ does not satisfy property $(gb)$. But since $T$ has
SVEP, $T$ satisfies a-Browder's theorem or equivalently, by
[\cite{AmouchM ZguittiH equivalence}, Theorem 2.2], $T$ satisfies
generalized a-Browder's theorem. Therefore by Lemma \ref{2.1} again,
$T$ satisfies property $(gb)$.

  \end{upshape}
\end{example}

Rashid gives in [\cite{Rashid}, Theorem 3.15] that if $T \in
\mathcal{B}(X)$ and $Q$ is a quasi-nilpotent operator that commute
with $T$, then
$$\sigma_{SBF_{+}^{-}}(T+Q)=\sigma_{SBF_{+}^{-}}(T).$$
  The next example show that this equality
does not hold in general.

\begin {example}${\label{2.13}}$ \begin{upshape} Let $Q$ denote the Volterra operator on the Banach space
$C[0,1]$ defined by
$$(Qf)(t) = \int_{0}^{t}f(s)\, \mathrm{d}s \makebox{\ \ \  for all } f \in
C[0,1] \makebox{\ and } t \in [0,1]. $$ $Q$ is injective and
quasi-nilpotent. Hence it is easy to see that $\mathcal {R}(Q^{n})$
is not closed for every $n \in \mathbb{N}$. Let $T = 0 \in
\mathcal{B}(C[0,1])$. It is easy to see that $TQ=0=QT$ and $0 \notin
\sigma_{SBF_{+}^{-}}(0) = \sigma_{SBF_{+}^{-}}(T)$, but $0 \in
\sigma_{SBF_{+}^{-}}(Q) = \sigma_{SBF_{+}^{-}}(0+Q) =
\sigma_{SBF_{+}^{-}}(T+Q)$. Hence $\sigma_{SBF_{+}^{-}}(T+Q) \neq
\sigma_{SBF_{+}^{-}}(T).$

  \end{upshape}
\end{example}

Rashid claims in [\cite{Rashid}, Theorem 3.16] that property $(gb)$
is stable under commuting quasi-nilpotent perturbations, but its
proof is rely on [\cite{Rashid}, Theorem 3.15] which, by Example
\ref{2.11}, is not always true. The following example shows
 that property $(gb)$ in general is not preserved under commuting quasi-nilpotent
 perturbations.

\begin {example}${\label{2.14}}$ \begin{upshape} Let $U: l_{2}(\mathbb{N}) \longrightarrow l_{2}(\mathbb{N})$
be the unilateral right shift operator defined by
$$U(x_{1},x_{2},\cdots )=(0,x_{1},x_{2}, \cdots ) \makebox{\ \ \  for all \ } (x_{n}) \in
l_{2}(\mathbb{N}).$$ Let $V: l_{2}(\mathbb{N}) \longrightarrow
l_{2}(\mathbb{N})$ be a quasi-nilpotent operator defined by
$$V(x_{1},x_{2},\cdots )=(0,x_{1},0,\frac{x_{3}}{3},\frac{x_{4}}{4} \cdots ) \makebox{\ \ \  for all \ } (x_{n}) \in
l_{2}(\mathbb{N}).$$ Let $N: l_{2}(\mathbb{N}) \longrightarrow
l_{2}(\mathbb{N})$ be a quasi-nilpotent operator defined by
$$N(x_{1},x_{2},\cdots )=(0,0,0,-\frac{x_{3}}{3},-\frac{x_{4}}{4} \cdots ) \makebox{\ \ \  for all \ } (x_{n}) \in
l_{2}(\mathbb{N}).$$ It is easy to verify that $VN=NV$. We consider
the operators $T$ and $Q$ defined by $T=U \oplus V$ and $Q=0 \oplus
N$, respectively. Then $Q$ is quasi-nilpotent and $TQ=QT$. Moreover,
$$\sigma(T) = \sigma(U) \cup \sigma(V) = \{\lambda \in
\mathbb{C}: 0 \leq |\lambda| \leq 1 \},$$
$$\sigma_{a}(T) = \sigma_{a}(U) \cup \sigma_{a}(V) = \{\lambda \in
\mathbb{C}: |\lambda| = 1 \} \cup \{0\},$$
$$\sigma(T+Q) = \sigma(U) \cup
\sigma(V+N) = \{\lambda \in \mathbb{C}: 0 \leq |\lambda| \leq 1 \}$$
and
$$\sigma_{a}(T+Q) = \sigma_{a}(U) \cup \sigma_{a}(V+N) = \{\lambda \in
\mathbb{C}: |\lambda| = 1 \} \cup \{0\}.$$ It follows that
$\Pi_{a}(T)=\Pi(T)= \varnothing$ and $\{0\}= \Pi_{a}(T+Q) \neq
\Pi(T+Q)= \varnothing.$ Hence by Lemma \ref{2.1}, $T+Q$ does not
satisfy property $(gb)$. But since $T$ has SVEP, $T$ satisfies
a-Browder's theorem or equivalently, by [\cite{AmouchM ZguittiH
equivalence}, Theorem 2.2], $T$ satisfies generalized a-Browder's
theorem. Therefore by Lemma \ref{2.1} again, $T$ satisfies property
$(gb)$.

  \end{upshape}
\end{example}

\begin {theorem}${\label{2.15}}$ Suppose that $T \in \mathcal{B}(X)$ obeys property $(gb)$ and that $Q \in \mathcal{B}(X)$ is a quasi-nilpotent operator commuting
with $T$. If $T$ is a-polaroid, then $T+Q$ obeys $(gb)$.

\end{theorem}

\begin{proof} Since $T$ satisfies property $(gb)$, by Lemma
\ref{2.1}, $T$ satisfies generalized a-Browder's theorem and
$\Pi(T)=\Pi_{a}(T)$. Hence $T+Q$ satisfies generalized a-Browder's
theorem. In order to show that $T+Q$ satisfies property $(gb)$, by
Lemma \ref{2.1} again, it suffices to show that
$\Pi(T+Q)=\Pi_{a}(T+Q)$. Since $\Pi(T+Q) \subseteq \Pi_{a}(T+Q)$ is
always true, one needs only to show that $\Pi_{a}(T+Q) \subseteq
\Pi(T+Q)$.

Let $\lambda \in \Pi_{a}(T+Q)=\sigma_{a}(T+Q) \backslash
\sigma_{LD}(T+Q) = \makebox{iso}\sigma_{a}(T+Q) \backslash
\sigma_{LD}(T+Q)$. Then by [\cite{Mbekhta-Muller 9}], $\lambda \in
\makebox{iso}\sigma_{a}(T)$. Since $T$ is a-polaroid, $\lambda \in
\Pi_{a}(T) = \Pi(T)$. Thus by [\cite{Zeng-Zhong-Wu}, Theorem 3.12],
$\lambda \in \Pi(T+Q)$. Therefore $\Pi_{a}(T+Q) \subseteq \Pi(T+Q)$,
and this completes the proof.
\end{proof}

\end{document}